\documentclass{amsart}
\usepackage{amsmath,latexsym,amssymb,amsfonts}
\usepackage{amscd,amssymb,epsfig}

\newtheorem{theorem}{Theorem}
\newtheorem{lemma}[theorem]{Lemma}

\newtheorem{corollary}[theorem]{Corollary}

\newtheorem{remark}[theorem]{Remark}

\title{Free Infinite Divisibility for Ultrasphericals}
\author{Octavio Arizmendi}
\address{Universit\"{a}t des Saarlandes, FR $6.1-$Mathematik, 66123 Saarbr\"{u}cken, Germany } 
\email{arizmendi@math.uni-sb.de}

\author{Serban T. Belinschi}
\address{Department of Mathematics \& Statistics,
Queen's University, and Institute of Mathematics ``Simion
Stoilow'' of the Romanian Academy;
Department of Mathematics and Statistics, 
Queen's University,
Jeffrey Hall, 
Kingston, ON K7L 3N6 CANADA} 
\email{sbelinsch@mast.queensu.ca}

\date{\today}
\begin{document}

\thanks{Work of S. Belinschi was supported by a Discovery
Grant from NSERC.\\ Octavio Arizmendi was supported by DFG-Deutsche Forschungsgemeinschaft Project SP419/8-1.}

\begin{abstract}We prove that the integral powers of the semicircular distribution are freely infinitely divisible. As a byproduct we get another proof of the free infnite divisibility of the classical Gaussian distribution.
\end{abstract}

\maketitle

\section{Introduction}

The semicircular distribution, with density $d\gamma_1(t)={(2\pi)}^{-1}\sqrt{4-t^2}{\bf 1}_{[-2,2]}\,dt$ is a very important probability measure, appearing as the distributional limit of the eigenvalues of large selfadjoint random matrices with independent entries. More important to us, this measure plays a central role in free probability theory since it arises from the free version of the central limit theorem, in particular, it is freely infinitely divisible (f. i. d. for short). 

In a different context, the semicircular distribution belongs to a family of measures called ultraspherical (or hyperspherical) distributions,
with density 
\begin{equation*}
d\gamma_n(t)=c_n(4-t^2)^{n-1/2}{\bf 1}_{[-2,2]}\,dt.
\end{equation*}
Here, the normalizing constant $c_n$ is simply the reciprocal of $\int_{-2}^2(4-t^2)^{n-1/2}\,dt$,
which, as a direct computation shows, equals $4^n\frac{1\cdot3\cdot5\cdots(2n-3)(2n-1)}{2\cdot4\cdot6\cdots(2n-2)\cdot(2n)}\pi$. 

The importance of this family comes, on one hand, from a geometrical point of view in the following way. Let $x=(x_1,x_2,...,x_{2n+2})$ be a random vector in $\mathbb{R}^{2n+2}$ with spherical symmetry and let $v=(v_1,\cdots,v_{2n+2})$ be \emph{any} fixed unit vector in $\mathbb{R}^{2n+2}$. If we denote by $\overline{x}:=\frac{x}{\rVert x\lVert}$ and by $\lambda$ the dot product $\lambda:=\overline{x}\cdot v$, then $2\lambda$ has a distribution which is ultraspherical $\gamma_n$ \cite{DiFr87,kingman}. Moreover, properly normalized, $\gamma_n$ converges to the Normal (or classical Gaussian) distribution  when $n$ tends to infinity, a result known as
Poincar\'{e}'s theorem; \cite{Poin}. For this reason, as explained by McKean \cite{Mc73},
one can think of the Wiener measure (all whose marginals are Gaussian) as the uniform distribution on an infinite
dimensional sphere. 

On the other hand, as observed in Arizmendi and Perez-Abreu \cite{APA}, the ultraspherical family contains all the Gaussian distributions with respect to the 5 fundamental independences in non-commutative probability, as classified by Muraki \cite{Mu02}.
Specifically, the left-boundary case $n=-1$ is the symmetric Bernoulli distribution appearing in the central limit theorem with respect to Boolean convolution (Speicher and Woroudi \cite{SpWo97}). Similarly, for $n=0$ we obtain the arcsine distribution, which
plays the same role in monotone and antimonotone convolutions as the Normal distribution does in classical
probability ( Muraki \cite{Mu97}).  As mentioned before, the case $n=1$, which is the semicircle distribution, and the right-boundary case $n=\infty$, which is the Normal distribution, correspond to the free and classical central limits theorems, respectively.

 From this observations, the free infinite divisibility of the ultraspherical distributions was
 considered in \cite{APA} as a mean to prove that the Normal distribution $\gamma_\infty$ is f. i. d. The authors proved using kurtosis arguments that $\gamma_n$ is not f. i. d. for $n<1$ and conjectured that it is f. i. d. for all $n\in[1,+\infty)$. The free infinite divisibility of the Normal distribution was proved in \cite{BBLS}. However the conjecture for the ultraspherical distributions remained open. 

In this paper we prove that $\gamma_n$ is f. i. d. for $n\in\mathbb N$. The method we
employ in our proof is similar to the method used in \cite{ABBL,BBLS}.
We will construct an inverse to the Cauchy transform $G_{\gamma_n}$
defined on the whole lower half-plane.
Then the Voiculescu transform $\phi_{\gamma_n}(1/z)=G_{\gamma_n}^{-1}(z)-(1/z)$
has an extension to the whole complex upper half-plane 
$\mathbb C^+$.

The following theorem from \cite{BVIUMJ}
allows us to conclude:

\begin{theorem}\label{bvid}
A Borel probability measure $\mu$ on the real line is $\boxplus$-infinitely divisible if and only if
its Voiculescu transform $\phi_\mu(z)$ extends to an analytic function $\phi_\mu\colon
\mathbb C^+\to\mathbb C^-$.
\end{theorem}

As a direct consequence, we get another proof of the free infinite divisibility of the Normal distribution.

The free infinite divisibility of other families including the semicircle and Normal distributions have been 
considered recently. A particularly important one are the so-called $q$-Gaussian distributions introduced by Bo\.{z}ejko and Speicher in \cite{BS,BKS}. This familiy was proved to be f. i. d. for all $q\in[0,1]$ in \cite{ABBL}.

Finally, let us mention that, at this point, all the proofs of the free infinite divisibility of the Normal distribution are purely analytical and, even though some combinatorial considerations were noted in \cite{BBLS}, a more conceptual explanation of this fact is still not known. However, connections between ultraspherical distributions with random matrices are known, see for example \cite{Blo}. Also, ultraspherical distributions appear in connection to quantum groups \cite{BaCo,BaGo,BaJu}. Moreover, the free analogue of ultraspherical laws was found in \cite{BaJu}. So, we think that studying this family in more detail may lead to a better understanding of the free infinite divisibility of the Normal distribution.

\section{Properties of $\gamma_n$}

We shall first collect some data about our objects. The Cauchy (or Cauchy-Stieltjes) transform
of a finite Borel measure $\mu$ on the real line is defined by
$$
G_\mu(z)=\int_\mathbb R\frac{1}{z-t}\,d\mu(t),\quad z\not\in\textrm{supp}(\mu),
$$
where $\textrm{supp}(\mu)$ denotes the topological support of $\mu$. This function is analytic on
its domain, and whenever $\mu$ is positive, it maps $\mathbb C^+$ into the lower half-plane
$\mathbb C^-$. The Cauchy transform of the semicircular distribution is known to have
a nice expression:
\begin{equation}
G_{\gamma_1}(z)=\frac{z-\sqrt{z^2-4}}{2},\quad z\in\mathbb C^+,
\end{equation}
where the square root branch is chosen so that $\lim_{z\to\infty}zG_{\gamma_1}(z)=1$.
One can easily verify that 
\begin{enumerate}
\item $G_{\gamma_1}(\overline{\mathbb C^+})=\overline{\mathbb D
\cap\mathbb C^-}$, and the correspondence is bijective. Moreover, 
\item $G_{\gamma_1}([-2,2])=
\{z\in\overline{\mathbb C^-}\colon|z|=1\}$,
\item $G_{\gamma_1}((-\infty,-2])=[-1,0)$, $G_{\gamma_1}([2,+\infty))=(0,1]$.
\end{enumerate}
It is remarkable that this function has an analytic extension through $\mathbb R\setminus[-2,2]$
which satisfies $G_{\gamma_1}(\overline{z})=\overline{G_{\gamma_1}(z)}$, and a
\emph{different} extension through $(-2,2)$ which satisfies the more convenient condition
that $G_{\gamma_1}(z)=\frac{z+\sqrt{z^2-4}}{2}$, where the square root is chosen
with the same condition as before. From now on whenever we write $G_{\gamma_1}$ we shall
refer to this extension: $G_{\gamma_1}\colon\mathbb C\setminus\{(-\infty,-2]\cup[2,+\infty)\}
\to\mathbb C^-$. It satisfies
\begin{enumerate}
\item $G_{\gamma_1}(\overline{\mathbb C^-})=\overline{\mathbb C^-\setminus\mathbb D}$,
\item $G_{\gamma_1}|_{\overline{\mathbb C^-}}((-\infty,-2])=(-\infty,-1]$, $G_{\gamma_1}|_{\overline{\mathbb C^-}}([2,+\infty))=[1,+\infty)$,
and
\item $G_{\gamma_1}([-2,2])=\{z\in\overline{\mathbb C^-}\colon|z|=1\}$.
\end{enumerate}
(The reader should note that the restriction of $G_{\gamma_1}$ to the upper half-plane
has a different extension to the complement of $[-2,2]$ than the restriction of $G_{\gamma_1}$
to the lower half-plane. On the lower half-plane, $G_{\gamma_1}$ behaves like the extension
through $\mathbb R\setminus[-2,2]$ of the function $\left.\frac{1}{G_{\gamma_1}}\right|_{
\mathbb C^+}$.)

It has been shown in \cite[Proposition 3.1]{demni} that for any $\lambda>0$ there exists a 
positive constant $d_\lambda$ so that 
\begin{equation}\label{2}
\int_{-2}^2\frac{1}{(z-t)^\lambda}(4-t^2)^{\lambda-1/2}\,dt=d_\lambda\left(
\int_{-2}^2\frac{1}{2\pi(z-t)}\sqrt{4-t^2}\,dt\right)^\lambda,\quad z\in\mathbb C^+.
\end{equation}
This will allow us to establish some useful functional and differential equations.
Note that $G^{(k)}_{\mu}(z)=k!\int\frac{(-1)^k}{(z-t)^{k+1}}\,d\mu(t)$, $k\in\mathbb N,
z\in\mathbb C^+$.

\begin{lemma}\label{lem3}
For any $z\in\mathbb C\setminus\{(-\infty,-2]\cup[2,+\infty)\},$ we have 
$$
G_{\gamma_n}(z)=Q_n(z^2)G_{\gamma_1}(z)+zP_n(z^{2}),
$$
where 
$$
Q_n(X)=2\pi c_n\sum_{j=0}^{n-1}\frac{(n-1)!}{j!(n-1-j)!}(-1)^j4^{n-1-j}X^j=2\pi c_n(4-X)^{n-1},$$
$$
P_n(X)=2\pi c_n\sum_{k=1}^{n-1}\left(\sum_{j=1}^{n-k}
\frac{(-1)^{j+k}4^{n-(j+k)}(n-1)!}{(j+k-1)!(n-j-k)!}C_{j-1}\right)X^{k-1},$$
and $c_n=\frac{n!}{2^n\pi(1\cdot3\cdot5\cdots(2n-3)(2n-1))}$. In particular, the Cauchy 
transform of $\gamma_n$ extends to an analytic function on
$\mathbb C\setminus\{(-\infty,-2]\cup[2,+\infty)\}.$

\end{lemma}
\begin{proof}
The proof is a simple direct computation:
\begin{eqnarray*}
G_{\gamma_n}(z) & = & 2\pi c_n\int_{-2}^2\frac{(4-t^2)^{n-1}}{2\pi(z-t)}\sqrt{4-t^2}\,dt\\
& = & 2\pi c_n\sum_{j=0}^{n-1}\frac{(n-1)!}{j!(n-1-j)!}(-1)^j4^{n-1-j}\int_{-2}^2\frac{t^{2j}}{2\pi(z-t)}\sqrt{4-t^2}\,dt\\
& = & 2\pi c_n\sum_{j=0}^{n-1}\frac{(n-1)!}{j!(n-1-j)!}(-1)^j4^{n-1-j}
\left[z^{2j}G_{\gamma_1}(z)-\sum_{k=0}^{j-1}C_kz^{2(j-k)-1}
\right]\\
& = & \left[2\pi c_n\sum_{j=0}^{n-1}\frac{(n-1)!}{j!(n-1-j)!}(-1)^j4^{n-1-j}z^{2j}\right]
G_{\gamma_1}(z)\\
&  & \mbox{}-2\pi c_n\sum_{j=1}^{n-1}\frac{(n-1)!}{j!(n-1-j)!}(-1)^j4^{n-1-j}\sum_{k=0}^{
j-1}C_kz^{2(j-k)-1}\\
& = & \left[2\pi c_n\sum_{j=0}^{n-1}\frac{(n-1)!}{j!(n-1-j)!}(-1)^j4^{n-1-j}z^{2j}\right]
G_{\gamma_1}(z)\\
&  & \mbox{}+2\pi c_n\sum_{k=1}^{n-1}\left(\sum_{j=1}^{n-k}
\frac{(-1)^{j+k}4^{n-(j+k)}(n-1)!}{(j+k-1)!(n-j-k)!}C_{j-1}\right)z^{2k-1}.
\end{eqnarray*}
\end{proof}

\begin{lemma}\label{lem2}
For any $z\in\mathbb C\setminus\{(-\infty,-2]\cup[2,+\infty)\},$ we have 
$$
(-1)^nG_{\gamma_n}^{(n-1)}(z)=\frac{d_n}{n!}G_{\gamma_1}(z)^n.
$$
\end{lemma}
\begin{proof}
The statement is a direct consequence of formula \eqref{2} and analytic continuation. Indeed,
for $\lambda =n$, the left hand side of \eqref{2} is, as noted above, the $(n-1)^{\rm st}$ 
derivative of $G_{\gamma_n}$, up to a multiplicative constant depending only in $n$.
Since $n\in\mathbb N, n>0$, it follows that 
$$
(-1)^nn!\int_{-2}^2\frac{1}{(z-t)^n}(4-t^2)^{n-1/2}\,dt=
\frac{d^n}{dz^n}\left(\int_{-2}^2\frac{1}{z-t}(4-t^2)^{n-1}\sqrt{4-t^2}\,dt\right).
$$
On the other hand, for any $k\in\mathbb N$, we have 
\begin{eqnarray*}
\frac{1}{2\pi}\int_{-2}^2\frac{t^{2k}}{z-t}\sqrt{4-t^2}\,dt & = & 
\frac{1}{2\pi}\int_{-2}^2\frac{(t-z)\left(\sum_{j=0}^{2k-1}t^{2k-j-1}z^j\right)}{z-t}\sqrt{4-t^2}\,dt\\
& & \mbox{}+
z^{2k}\frac{1}{2\pi}\int_{-2}^2\frac{1}{z-t}\sqrt{4-t^2}\,dt\\
& = & z^{2k}G_{\gamma_1}(z)-\sum_{j=0}^{k-1}C_{j}z^{2(k-j)-1},
\end{eqnarray*}
where $C_j$ is the $j^{\rm th}$ Catalan number, the $2j^{\rm th}$ moment of the standard
semicircular distribution $\gamma_1$.
We shall record this equality for future reference:
\begin{equation}\label{3}
\frac{1}{2\pi}\int_{-2}^2\frac{t^{2k}}{z-t}\sqrt{4-t^2}\,dt = z^{2k}G_{\gamma_1}(z)-\sum_{j=0}^{k-1}C_{j}z^{2(k-j)-1},\quad z\not\in(-\infty,-2]\cup[2,+\infty).
\end{equation}

Since $n\in\mathbb N,n>0$, it follows that
$$
\frac{d^n}{dz^n}\left(\int_{-2}^2\frac{1}{z-t}(4-t^2)^{n-1}\sqrt{4-t^2}\,dt\right)=
\frac{d^n}{dz^n}\left(Q_n(z^2)G_{\gamma_1}(z)+zP_n(z^2)\right).
$$
Here $P_n$ and $Q_n$ are polynomials obtained in the previous lemma.
Since the right hand term has a continuation from $\mathbb C^+$ to
$\mathbb C\setminus\{(-\infty,-2]\cup[2,+\infty)\}$, we have proved our statement.
\end{proof}

It might be useful to see the Cauchy transforms of the first two ultrasphericals:
$$
G_{\gamma_2}(z)=2\pi c_2(-z^2+4)G_{\gamma_1}(z)+2\pi c_2z,
$$ and
$$
G_{\gamma_3}(z)=2\pi c_3(z^4-8z^2+16)G_{\gamma_1}(z)+2\pi c_3z(-z^2+7).
$$
We shall need also the following recurrence relation:
\begin{equation}\label{4}
G_{\gamma_{n+1}}(z)=\frac{n+1}{2(2n+1)}(4-z^2)G_{\gamma_n}(z)+\frac{n+1}{2(2n+1)}z.
\end{equation}
\begin{remark}\label{rmk4}
The above lemma has an interesting consequence for the behaviour of $G_{\gamma_n}$ along the
imaginary axis. It follows trivially from the symmetry of $\gamma_n$ that $G_{\gamma_n}(i\mathbb R)\subseteq i\mathbb R$. This implies that $G_{\gamma_n}'(iy)\in\mathbb R$ for
all $y\in\mathbb R$.
Moreover, as it will be shown below (independent from this
remark), $G_{\gamma_n}'(iy)\neq0$ whenever $G_{\gamma_n}(iy)\in\mathbb C^-$. Since
this is true for all $y>0$, it follows that $G'_{\gamma_n}(iy)$ has constant sign along
the positive half of the imaginary axis. It can be found through direct computation that 
for large $y$ we have $G'_{\gamma_n}(iy)>0,$ so it follows that $G'_{\gamma_n}(iy)$
remains positive for all $y>0$. We claim that in fact this must hold for \emph{all} $y\in
\mathbb R$. Indeed, otherwise it would be necessary that $G'_{\gamma_n}(iy_0)=0$
for some $y_0\le0$. However, since $\Im G_{\gamma_n}(iy_0)<0,$ this would contradict
the lemma below. In particular, this allows us to conclude that $G_{\gamma_n}$
maps bijectively $i\mathbb R$ onto $i(-\infty,0)$ (with $\lim_{y\to-\infty}G_{\gamma_n}(iy)=
-i\infty$, $\lim_{y\to+\infty}G_{\gamma_n}(iy)=0$.)
\end{remark}

\section{Main Results}

\begin{lemma}\label{deriv}
If $G_{\gamma_n}(z)\in\mathbb C^-$, then $G_{\gamma_n}'(z)\neq0$. Moreover, the extension 
of $G_{\gamma_n}|_{\mathbb C^+}$ to the real line is continuous, bounded and injective.
\end{lemma}
\begin{proof}
For $n=1$, the statement is trivial. For $n=2$, 
it follows directly from Lemma \ref{lem2} that $G'_{\gamma_2}(z)\neq0$ for all $z\in
\mathbb C\setminus\{(-\infty,-2]\cup[2,+\infty)\}$.
For general $n$, we note that 
\begin{eqnarray*}
G_{\gamma_{n+1}}'(z) & = & -c_{n+1}\int_{-2}^2\frac{(4-t^2)^{n+\frac12}}{(z-t)^2}\,dt\\
& = & \left.-c_{n+1}\frac{(4-t^2)^{n+\frac12}}{z-t}\right|_{-2}^2+c_{n+1}\left(n+\frac12
\right)\int_{-2}^2\frac{-2t(4-t^2)^{n-\frac12}}{z-t}\,dt\\
& = & \frac{c_{n+1}}{c_n}\left(2n+1
\right)c_n\int_{-2}^2\frac{(z-t-z)(4-t^2)^{n-\frac12}}{z-t}\,dt\\
& = & \frac{c_{n+1}}{c_n}\left(2n+1
\right)\left[1-zG_{\gamma_n}(z)\right]\\
& = & \frac{n+1}{2}\left[1-zG_{\gamma_n}(z)\right],\quad z\in\mathbb C^+.
\end{eqnarray*}
Analytic continuation guarantees that this relation holds on the whole
domain $\mathbb C\setminus\left\{(-\infty,-2]\cup[2,+\infty)\right\}$.
Thus, the equality $G_{\gamma_{n+1}}'(z)=0$ implies necessarily 
$G_{\gamma_n}(z)=\frac1z$. If $z\in\mathbb C^+\cup[-2,2]$, this equality implies that
$\gamma_n=\delta_0$, an obvious contradiction. If $z\in\mathbb C^-$, then we must recall 
from \eqref{4} that
$$
z\left[\frac{2(2n+1)}{n+1}G_{\gamma_{n+1}}(z)-z\right]\frac{1}{4-z^2}=1.
$$
This is equivalent to
$$
\frac{2(2n+1)}{n+1}G_{\gamma_{n+1}}(z)=\frac4z\in\mathbb C^+.
$$
This contradicts the hypothesis of our lemma.

We shall now prove that the continuous extension of $G_{\gamma_n}$ from the upper half-plane
to the real line is bounded and continuous, and that $G_{\gamma_n}|_{\mathbb R\cup\{\infty\}}$ 
is  injective. Continuity of $G_{\gamma_n}$ follows trivially from the continuity of $G_{\gamma_1}$ and Lemma \ref{lem3}.
Recall that $G_{\gamma_1}(\mathbb C^+\cup\mathbb R\cup\{\infty\})=\overline{
\mathbb D\cap\mathbb C^-},$ so boundedness is obvious for $n=1$. This together with
Lemma \ref{lem3} guarantees that for any fixed $n\in\mathbb N,R>0$, the set 
$G_{\gamma_n}(\{\mathbb C^+\cup\mathbb R\}\cap\{z\in\mathbb C\colon|z|\leq R\})$ is
bounded. On the other hand, for $R>0$ large enough, we know that $G_{\gamma_n}
(\{\mathbb C\cup\{\infty\}\}\setminus\{z\in\mathbb C\colon|z|\leq R\})$ is a bounded
neighbourhood of zero. This proves that $G_{\gamma_n}$ is bounded on 
$\mathbb C^+\cup\mathbb R$ for all $n\in\mathbb N$.

As seen above, $G_{\gamma_n}'(z)\neq0$ for all $z\in\mathbb C^+$. The inequality
$G_{\gamma_n}'(x)<0$ for all $x\in\mathbb R\setminus[-2,2]$ is a trivial consequence of
the definition of the Cauchy-Stieltjes transform and the fact that $\gamma_n$ is supported by
$[-2,2]$. Equation \eqref{4} guarantees that $G_{\gamma_n}(\pm2)=\pm\frac{n}{2n-1}.$
Employing again Lemma \ref{lem3}, we obtain for $x\in(-2,2)$ that
$\Re G_{\gamma_n}(x)=\pi c_n x[(4-x^2)^{n-1}+P_n(x^2)]$ and
$\Im G_{\gamma_n}(x)=-\pi c_n(4-x^2)^{n-\frac12}$. The derivative of the imaginary part
is, by direct computation, strictly positive on $(0,2)$, strictly negative on $(-2,0)$ and 
equal to zero in $0,\pm2$ (as \emph{real} function in $\pm2$). This already guarantees the 
injectivity of $G_{\gamma_n}$ on the intervals $(-\infty,0]$ and $[0,+\infty)$. We still need to
show that $G_{\gamma_n}([-2,0))\cap G_{\gamma_n}((0,2])=\varnothing$. The identity 
principle together with the analyticity of $G_{\gamma_n}$ on $(-2,2)$ guarantee that, if
this intersection is nonempty, then it can only be a discrete set. Recalling that $\gamma_n$
is a symmetric measure, it follows that  a point $w\in G_{\gamma_n}([-2,0))\cap G_{\gamma_n}((0,2])$ must be of the form $w=G_{\gamma_n}(r)=G_{\gamma_n}(-r)$ for some $r\in(0,2)$. 
In particular, $w\in-i(0,+\infty)$. Using equation \eqref{4}, we find that if $\Re G_{\gamma_n}(r)
=0$, then $\Re[(4-r^2)G_{\gamma_{n-1}}(r)+r]=0$, so that $\Re G_{\gamma_{n-1}}(r)
=-\frac{r}{4-r^2}<0$. Since $G_{\gamma_{n-1}}(2)=\frac{n-1}{2n-3}>0$, for $G_{\gamma_{n-1}}$ we can find a new $r\in(0,2)$
so that $G_{\gamma_{n-1}}(r)\in i\mathbb R$. Repeating, we find that there exists 
some $r\in(0,2)$ so that $\Re G_{\gamma_2}(r)=0$. This implies that
$r+(4-r^2)\Re G_{\gamma_1}(r)=0$, so $r+(4-r^2)\frac{r}2=0$. Simplifying $r$, we obtain
$2+4-r^2=0$, so $r=\sqrt6>2$, an obvious contradiction. Thus $G_{\gamma_n}$ is injective 
on $\mathbb R$, as claimed in our lemma.
\end{proof}

A very important consequence of the above lemma is that $G_{\gamma_n}$ is injective
on the whole upper half-plane (see \cite{CP}).

\begin{theorem}\label{Thm6}
For each $n\in\mathbb N,n>0$, there exists a simply connected domain $D_n$ so that 
$\mathbb C^+\cup(-2,2)\subset D_n\subseteq\mathbb C\setminus\{(-\infty,-2]\cup[2,+\infty)\}$
and $G_{\gamma_n}\colon D_n\to\mathbb C^-$ is a conformal map.
In particular, $\gamma_n$ is freely infinitely divisible for each $n\ge1$.
\end{theorem}
Remark \ref{rmk4} gives an indication of how $D_n$ must look like: it must be symmetric 
with respect to the imaginary axis and we must have $i\mathbb R\subset D_n$ for all $n$.
We shall make all this precise in our proof below. 
\begin{proof}

We observe first that, according to Lemma \ref{deriv}, $G_{\gamma_n}$ is injective on 
$\mathbb C^+\cup(-2,2)$, and according to Remark \ref{rmk4}, it is injective on the
imaginary line. Moreover, the continuous extension of $G_{\gamma_n}|_{\mathbb C^+}$ to 
$\mathbb R\cup\{\infty\}$ is a simple closed curve which cuts the imaginary axis exactly twice,
namely in $G_{\gamma_n}(0)$ and $0=G_{\gamma_n}(+i\infty)$. Thus, $G_{\gamma_n}$
is injective on a small enough complex neighbourhood of $\mathbb C^+\cup(-2,2)\cup i\mathbb R$.

Recall that the extension of $G_{\gamma_1}$ to the lower
half-plane through $(-2,2)$ satisfies the condition
$$
\lim_{\stackrel{z\to\infty}{z\in\overline{\mathbb C^-}}}\frac{G_{\gamma_1}(z)}{z}=1.
$$
Using Lemma \ref{lem3} we immediately observe a similar behaviour of $G_{\gamma_n}$.
Indeed, $Q_n(z^2)$ has degree $2n-2$ and $zP_n(z^2)$ has degree $2n-3$.
In particular, the asymptotics at infinity of $G_{\gamma_n}$ in the lower half-plane 
is of order $z^{2n-1}$. Considering the extension of $G_{\gamma_n}|_{\mathbb C^-}$
to the upper half-plane through the complement of $[-2,2]$ we obtain a map
which is $2n-1$-to-1 on a neighbourhood of infinity and which preserves the real and imaginary
lines close to infinity.

For each $t>0$, define $\eta_t$ to be the open segment that 
unites $-it\in i(-\infty,0)$ and $\frac{n}{2n-1}+t\in(\frac{n}{2n-1},+
\infty).$ It is clear that $\bigcup_{t>0}\eta_t=\{z\in\mathbb C^-\colon
\Re z>0\}.$ For each fixed $t>0$ we shall prove that 
there exists a unique simple path $a_t$ uniting
$G^{-1}_{\gamma_n}(-it)$ to a point in $(2,+\infty)$ so that
$G_{\gamma_n}(a_t)=\eta_t$.

The above remarks regarding the injectivity of $G_{\gamma_n}$ on a neighbourhood of the
imaginary line guarantee that $G_{\gamma_n}^{-1}$ is well-defined
on $i(-\infty,0)$, and so $G^{-1}_{\gamma_n}
(-it)$ is the unique choice for starting $a_t$. This choice imposes
that $a_t$ moves into the right half of the complex plane 
(recall that $G_{\gamma_n}'$ is positive on the imaginary line). 
Let us now extend $a_t$ under the condition that $G_{\gamma_n}(a_t)
\subseteq\eta_t$. We claim that we can extend $a_t$ {\em uniquely} all the way to having  $G_{\gamma_n}(a_t)=\eta_t$. Moreover, 
$a_t$ will eventually enter the lower half-plane and approach 
$(2,+\infty)$ from the lower right quadrant.
Indeed, assume towards contradiction that this
is not the case. There are two possible obstacles to this extension: first obstacle is a 
critical point for $G_{\gamma_n}$ along $a_t$, and the second is the possibility that 
the extension of $a_t$ leaves $\{z\in\mathbb C\colon\Re z>0\}\setminus
[2,+\infty)$.
Let us discuss the first obstacle first: if there is a point $c_0$ to which $a_t$ has been extended
by continuity so that $G_{\gamma_n}'(c_0)=0$, then, as long as $a_t$ 
has not left $\{z\in\mathbb C\colon\Re z>0\}\setminus
[2,+\infty)$, we have a direct contradiction to
Lemma \ref{deriv}, since $G_{\gamma_n}(c_0)\in\mathbb C^-$ by construction.
Thus, we need next to discard the case when $a_t$ leaves 
$\{z\in\mathbb C\colon\Re z>0\}\setminus[2,+\infty)$. It is obvious 
that it cannot leave it through $i\mathbb R$. Since both 
$G_{\gamma_n}|_{\mathbb C^+}$ and $G_{\gamma_n}|_{\mathbb C^-}$
map $\mathbb R\setminus[-2,2]$ to $\mathbb R$, it is equally obvious that $a_t$ cannot leave through $[2,+\infty)$, either from above 
(i.e. $\mathbb C^+$) or below (i.e. $\mathbb C^-$). Since 
$\lim_{z\to\infty}G_{\gamma_n}(z)=\infty$ when the limit is taken with 
values of $z$ in the lower half-plane, and $\lim_{z\to\infty}
G_{\gamma_n}(z)=0$ when the limit is taken with values of $z$ in the 
upper half-plane, $a_t$ cannot leave $\{z\in\mathbb C\colon\Re z>0\}
\setminus[2,+\infty)$ through infinity either.  
So $a_t$ cannot leave this set at all.
We conclude that $a_t$ can indeed be extended so that $G_{\gamma_n}
\colon a_t\mapsto\eta_t$ bijectively.

Recalling that there is a simply connected neighbourhood $V$ of 
$\mathbb C^+\cup(-2,2)\cup i\mathbb R$ on which $G_{\gamma_n}$
is injective, we find the simply connected open set $G_{\gamma_n}(V)
\subset\mathbb C^-$, which contains $i(-\infty,0)$, on which we can
uniquely define an analytic map $G_{\gamma_n}^{-1}$ which satisfies
$G_{\gamma_n}\circ G_{\gamma_n}^{-1}=\text{Id}_{G_{\gamma_n}(V)}$ and
$G_{\gamma_n}^{-1}\circ G_{\gamma_n}=\text{Id}_V$. Symmetry of
$\gamma_n$ implies that, first, we may assume $V$ to be symmetric
with respect to $i\mathbb R$ and, second, that $G_{\gamma_n}^{-1}$
must also satisfy $G_{\gamma_n}^{-1}(u+iv)=G_{\gamma_n}^{-1}(-u+iv)$.
From Lemma \ref{lem3} it follows that $z=Q_n(G_{\gamma_n}^{-1}(z)^2)
G_{\gamma_1}(G_{\gamma_n}^{-1}(z))+G_{\gamma_n}^{-1}(z)P_n(
G_{\gamma_n}^{-1}(z)^2)$; since $P_n,Q_n$ are polynomials and 
$G_{\gamma_1}$ is defined on $\mathbb C\setminus\left\{(-\infty,-2]\cup
[2,+\infty)\right\}$, this relation holds for all $z\in G_{\gamma_n}(V)
$. This equation will allow us to extend $G_{\gamma_n}^{-1}$ to all
of $\mathbb C^-$. Indeed, for any point $w\in\mathbb C^-$ there exists
a unique path $\eta_t$ so that $w\in\eta_t$. We claim that 
$G_{\gamma_n}^{-1}$ extends uniquely along $\eta_t$: since we have
shown the existence and uniqueness of a path $a_t\subset
\{z\in\mathbb C\colon\Re z>0\}\setminus[2,+\infty)$ so that
$G_{\gamma_n}$ maps $a_t$ bijectively onto $\eta_t$, 
it is clear that we can define $G_{\gamma_n}^{-1}$ on $\eta_t$
and with values in $a_t$. In addition, we recall
that, $\eta_t$ being in the lower half-plane and $a_t$ in 
$\{z\in\mathbb C\colon\Re z>0\}\setminus[2,+\infty)$, Lemma \ref{deriv}
guarantees that $G_{\gamma_n}(w)=
Q_n(w^2)G_{\gamma_1}(w)+wP_n(w^2)$ has a nonzero derivative at each
$w\in a_t$, and thus, by the analytic inverse function theorem, 
$G_{\gamma_n}^{-1}$ can be defined as an analytic map on a 
neighbourhood of $\eta_t$. Thus, since $\bigcup_{t>0}\eta_t=
\{z\in\mathbb C^-\colon\Re z>0\}$, we have extended our map
$G_{\gamma_n}^{-1}$ to $V\cup\{z\in\mathbb C^-\colon\Re z>0\}$.
We extend it to $\mathbb C^-$ by the formula 
$G_{\gamma_n}^{-1}(u+iv)=G_{\gamma_n}^{-1}(-u+iv)$. This completes
our proof, with $D_n=\bigcup_{t>0}a_t\cup i\mathbb R \cup \bigcup_{t>0}(-\overline{a_t})$.

\end{proof}

\begin{remark}[free divisibility indicator]
In terms of the free divisibility indicator $\phi(\mu)$ defined in \cite{BeNi} the last theorem says that $\phi(\gamma_n)\geq1$, for $n\in\mathbb N$. Using results in \cite{APA} it is easily shown that $\phi(\gamma_n)\leq\frac{2n+1}{n+2}.$
\end{remark}

The free infinite divisibility of the Gaussian distribution then follows.
\begin{corollary}
 The classical Gaussian distribution $d\gamma_\infty(t):={(2\pi)}^{-1}e^{-t^2/2}dt$ is freely infnitely divisible.
\end{corollary}
\begin{proof}
The class of f. i. d. 's is closed with respect to the weak converngence. Since, by
Poincar\'{e}'s theorem \cite{Poin}, the distributions $\gamma_n$, properly normalized, converge weakly
to a Gaussian distribution, we get the result by Theorem \ref{Thm6}.
\end{proof}

Finally, from our main result we see that some beta distributions are freely infinitely divisible. Recall that a beta distribution is given by its density function
\begin{equation*}
d_{Be(\alpha,\beta)}(t)=\frac{1}{B(\alpha,\beta)}(t)^{\alpha-1}(1-t)^{\beta-1}{\bf 1}_{[0,1]}\,dt.
\end{equation*}

\begin{corollary} For $n\in\mathbb N$ the following families are freely infinitely divisible:

 (1) The beta distributions $Beta(n+1/2,n+1/2)$.

 (2) The beta distributions $Beta(1/2,n+1/2)$ and $Beta(n+1/2,1/2)$ .
\end{corollary}
\begin{proof}
 1) If $X\sim\gamma_n$ then $1/4(1-X)$ has a distribution $Beta(n+1/2,n+1/2)$. Since free infinite divisibility is preserved under affine transformations we see that $Beta(n+1/2,n+1/2)$ is f. i. d.

 2) It was proved recently in \cite{AHS} that the square of any symmetric f. i. d. is also f. i. d. Now if $X\sim\gamma_n$ then $Y=1/4X^2$ has a distribution $Beta(1/2,n+1/2)$ which then is f. i. d. Finally, for the same reason as in (1), the distribution of $1-Y$ shall be f. i. d. which is a $Beta(n+1/2,1/2)$.
\end{proof}

\end{document}